\documentclass[12pt]{amsart}
\usepackage{url}
\usepackage{amsmath,amsxtra,amssymb,latexsym,epsfig,amscd,amsthm,fancybox,epsfig}
\usepackage[mathscr]{eucal}
\usepackage{graphicx}
\usepackage{multicol,xcolor}
\usepackage{epsfig} % for postscript graphics files
\usepackage{epstopdf}
\usepackage{cases}
\usepackage{color}
\usepackage{hyperref}
\usepackage{bigints}
\usepackage{enumitem}
\usepackage{subcaption}

\usepackage{natbib}
\usepackage{hyperref}
\hypersetup{
    colorlinks=true,
    linkcolor=blue,
    filecolor=magenta,
    urlcolor=cyan,
    citecolor=blue,
}

\let\amsthmproof\proof
\let\amsthmqed\qed
\let\amsthmqedsymbol\qedsymbol

\let\proof\amsthmproof
\let\qed\amsthmqed
\let\qedsymbol\amsthmqedsymbol

\newcommand{\tab}{\hspace{1em}}

\usepackage{epstopdf}
\usepackage{cases}
\usepackage{url, calc}

\usepackage[body={6.5in,9.0in},left=1in,top=1in,centering]{geometry}

\newtheorem {asp}{Assumption}[section]
\newtheorem{lm}{Lemma}[section]
\newtheorem{prop}{Proposition}[section]
\newtheorem{deff}{Definition}

\newtheorem{thm}{Theorem}
\newtheorem{rem}{Remark}
\newtheorem{exam}{Example}[section]
\numberwithin{equation}{section}

\DeclareMathOperator{\dist}{dist}
\newcommand{\eps}{\varepsilon}

\newcommand{\M}{\mathcal{S}}

\newcommand{\cN}{\mathcal{N}}
\newcommand{\cM}{\mathcal{M}}

\newcommand{\E}{\mathbb{E}}
\newcommand{\PP}{\mathbb{P}}

\DeclareMathOperator{\Conv}{Conv}

\newcommand{\N}{{\mathbb{Z}}_+}

\newcommand{\Z}{\mathbb{Z}}

\newcommand{\K}{\mathcal{K}}

\newcommand{\R}{\mathbb{R}}

\numberwithin{equation}{section}
\newcommand{\1}{\boldsymbol{1}}
\newcommand{\0}{\boldsymbol{0}}
\allowdisplaybreaks

\newcommand{\wtd}{\widetilde}

\newcommand{\BX}{\mathbf{X}}
\newcommand{\bx}{\mathbf{x}}
\newcommand{\bp}{\mathbf{p}}
\newcommand{\by}{\mathbf{y}}

\newcommand{\wdt}{\widetilde}

\newcommand{\Lom}{{\mathcal L}}

\newcommand{\bed}{\begin{equation}}
\newcommand{\eed}{\end{equation}}
\newcommand{\bea}{\bed\begin{array}{rl}}
	\newcommand{\eea}{\end{array}\eed}

\newcommand{\barray}{\begin{array}{ll}}
	\newcommand{\earray}{\end{array}}

\def\bar{\overline}
\def\hat{\widehat}
\def\a.s{\text{\;a.s.\;}}
\def\supp{\text{supp\,}}

\def\bdelta{\boldsymbol{\delta}}

\title{Random switching in an ecosystem with two prey and one predator}

\author[A. Hening]{Alexandru Hening }
\address{Department of Mathematics\\
Texas A\&M University\\
Mailstop 3368\\
College Station, TX 77843-3368\\
United States
}
\email{ahening@tamu.edu}

\author[D.H. Nguyen]{Dang H. Nguyen }
\address{Department of Mathematics \\
University of Alabama\\
345 Gordon Palmer Hall\\
Box 870350 \\
Tuscaloosa, AL 35487-0350 \\
United States}
\email{dangnh.maths@gmail.com}

\author[N. Nguyen]{ Nhu Nguyen }
\address{Department of Mathematics \\
University of Connecticut\\
341 Mansfield Road U1009\\
Storrs, Connecticut 06269-1009\\
United States}
\email{nguyen.nhu@uconn.edu }

\author[H. Watts]{Harrison Watts }
\address{Department of Mathematics \\
University of Alabama\\
345 Gordon Palmer Hall\\
Box 870350 \\
Tuscaloosa, AL 35487-0350 \\
United States}
\email{hrwatts@crimson.ua.edu}

\begin{document}
\begin{abstract}

In this paper we study the long term dynamics of two prey species and one predator species. In the deterministic setting, if we assume the interactions are of Lotka-Volterra type (competition or predation), the long term behavior of this system is well known. However, nature is usually not deterministic. All ecosystems experience some type of random environmental fluctuations. We incorporate these into a natural framework as follows. Suppose the environment has two possible states. In each of the two environmental states the dynamics is governed by a system of Lotka-Volterra ODE. The randomness comes from spending an exponential amount of time in each environmental state and then switching to the other one. We show how this random switching can create very interesting phenomena. In some cases the randomness can facilitate the coexistence of the three species even though coexistence is impossible in each of the two environmental states. In other cases, even though there is coexistence in each of the two environmental states, switching can lead to the loss of one or more species. We look into how predators and environmental fluctuations can mediate coexistence among competing species.

\end{abstract}

\maketitle
	\section{Introduction}
An important question in ecology is the relationship between complexity and stability. In particular, ecologists have been interested in whether predators can help facilitate coexistence or whether they are always detrimental to species diversity. Since the important work by \cite{P66} it has been clear that predators play a fundamental role in species diversity. There are experimental studies which show that the removal of predators can lead to the extinctions of various species. Other studies have shown the opposite effect, namely, that introducing a predator does not help mediate coexistence or that the addition of the predator leads to fewer species coexisting. In this paper we are interested in exploring these phenomena in the setting of Lotka-Volterra (LV) dynamics. The dynamics of two competing species is well-known in this setting - it can lead to coexistence where both species persist, competitive exclusion where one species is dominant and drives the other one extinct or to bistability where, depending on the initial conditions, one species persists and one goes extinct. There have been numerous studies which looked at how the introduction of a predator changes the long term outcome of two competitors, see work by \cite{HV83, TA83, S97}.

Every natural system experiences unpredictable environmental fluctuations. In the ecological setting, these environmental fluctuations will change the way species grow, die, and interact with each other. It is therefore key to include environmental fluctuations in the mathematical framework when trying to determine species richness. Sometimes the deterministic dynamics can predict certain species going extinct. However, if one adds the effects of a random environment extinction might be reversed into coexistence. In other cases deterministic systems that coexist become extinct once one takes into account the random environmental fluctuations. One of fruitful way of introducing randomness is by modelling the populations as discrete or continuous time Markov processes and analyzing the long-term behavior of these processes \citep{C82, CE89, C00, ERSS13, EHS15,  LES03, SLS09, SBA11, BS09, BHS08, B18, HNC20}.

There are many ways in which one can model the environmental fluctuations that affect an ecological system. One way is by going from ordinary differential equations (ODE) to stochastic differential equations (SDE). This amounts to saying that the various birth, death and interaction rates in an ecosystem are not constant, but fluctuate around their average values according to some white noise. There is now a well established general theory of coexistence and extinction for these systems \citep{SBA11,HN16,HNC20}. However, this way of modelling environmental fluctuations can sometimes seem artificial in an ecological setting. In certain ecosystems, it makes more sense to assume that when the environment changes, the dynamics also changes significantly. In a deterministic setting this can be modelled by periodic vector fields which can be interpreted to mimic seasonal fluctuations. In the random setting, these types of fluctations are captured by piecewise deterministic Markov processes (PDMP) - see the work by \cite{D84} for an introduction to these stochastic processes. In a PDMP, the environment switches between a fixed finite number of states to each of which we associate an ODE. In each state the dynamics is given by the flow of its associated ODE. After a random time, the environment switches to a different state, and the dynamics is governed by the ODE from that state.

Recently there have been some important results for two-species ecosystems that showcased how the switching behavior of PDMP can create novel ecological phenomena. The first set of results is for a two-species competitive LV model. In \cite{BL16, HN20} the authors show that the random switching between two environments that are both favorable to the same species, e.g. the favored species is dominant and persists and the unfavored species goes extinct, can lead to the extinction of this favored species and the persistence of the unfavored species, to the coexistence of the two competing species, or to bistability. This is extremely interesting as it relates to the competitive exclusion principle \citep{V28, H60, L70}, a fundamental principle of ecology, which says in its simplest form that when multiple species compete with each other for the same resource, one competitor will win and drive all the others to extinction. Nevertheless, it has been observed in nature that multiple species can coexist despite limited resources.  \cite{H61} gave a possible explanation by arguing that variations of the environment can keep species away from the deterministic equilibria that are forecasted by the competitive exclusion principle. The PDMP example from \cite{BL16, HN20} shows how the switching can save species from extinction, even though in each fixed environment, the same species is dominant. The second result looks at the classical predator-prey LV model. In \cite{HS17} the authors study a system that switches randomly between two deterministic classical Lotka-Volterra predator-prey systems. Even though for each deterministic predator-prey system the predator and the prey densities form closed periodic orbits, it is shown in \cite{HS17} that the switching makes the system leave any compact set. Moreover, in the switched system, the predator and prey densities oscillate between $0$ and $\infty$. These two sets of results show that random switching can radically change the dynamics of the system, and create new, possibly unexpected, long term results.

For three-species LV systems, the classification of the dynamics is incomplete in the deterministic setting. In the setting of SDE an almost complete classification appears in \cite{HNS21}. Not much is known for the dynamics of three-species systems in the PDMP setting. We hope that this paper will provide valuable results both phenomenologically, by showcasing some counterintuitive results, and mathematically, by developing new tools for the analysis of the ergodic properties of PDMP.

The deterministic dynamics are given by
\begin{equation}\label{2prey}
\begin{aligned}
\frac{d X_1}{dt}(t)&=X_1(t)[r-X_1(t)-b_1 X_2(t)-c_1 X_3(t)],\\
	\frac{d X_2}{dt}(t)&=X_2(t)[r-X_2(t)-b_2 X_1(t)-c_2 X_3(t)],\\
	\frac{d X_3}{dt}(t)&=  X_3(t)[e_1X_1(t)+e_2X_2(t)-d].
\end{aligned}
\end{equation}
Here $X_1(t), X_2(t)$ are the densities of the two prey species at time $t\geq 0$ while $X_3(t)$ is the density of the generalist predator at time $t\geq 0$. For simplicity we assume that the per-capita growth rates of both prey species are equal and given by $r>0$ and that the per-capita intraspecies competition are both equal to $1$. The per-capita interspecies competition rate of species $j$ on species $i$ is given by $b_i>0$ where $i,j\in\{1,2\}.$ The predator dies, when there is no prey, at the per-capita rate $d>0$, the predation rates on species $1$ and $2$ are given by $c_1, c_2>0$ and the quantities $e_1, e_2>0$ measure how efficient the predator is at using up the predated species.
We will sometimes write \eqref{2prey} in the more compact form
\begin{equation}\label{2prey_short}
\frac{dX_i}{dt}(t)=X_i(t)f_i(\BX(t)), i=1, 2, 3,
\end{equation}
where $\BX := (X_1, X_2, X_3), f_1(\bx):= r-x_1-b_1 x_2-c_1 x_3, f_2(\bx)= r-x_2-b_2 x_1-c_2 x_3, f_3(\bx)=e_1x_1+e_2x_2-d$.
In the absence of the predator ($X_3=0$) if we have
\begin{equation}\label{e:2d_co1}
b_1<1, b_2<1
\end{equation}
then the coexistence of $(X_1, X_2)$ is impossible (except for a stable manifold of dimension 1)- one species will go extinct \citep{TA83,  S97}.
However,  if one assumes additionally that
\begin{equation}\label{e:2d_co2}
\begin{aligned}
e_1r&>d, \\
e_2r&>d,\\
r-\frac{d}{e_1}b_2-\left(r-\frac{d}{e_1}\right)\frac{c_2}{c_1}&>0,\\
r-\frac{d}{e_2}b_1-\left(r-\frac{d}{e_2}\right)\frac{c_1}{c_2}&>0
\end{aligned}
\end{equation}
then the three species will coexist \citep{TA83,  S97}. This shows that it is possible for the predator to mediate coexistence in this setting.

We next explain how the switching is introduced. We assume there are a two environmental states $\M:=\{1,2\}$. We note that our theoretical analysis works for any finite number of environmental states. The environmental state at time $t\geq 0$ will be given by $\xi(t)\in \M$.  We suppose that the coefficients $c_1, c_2, e_1, e_2$ which capture the interaction between the predator and the two prey species, are different in the two environmental states. As a result we will have coefficients $c_1(j), c_2(j), e_1(j), e_2(j)$ if the environment is in state $j$.

The dynamics becomes
\begin{equation}\label{2prey_pdmp}
\frac{dX_i}{dt}(t)=X_i(t)f_i(\BX(t),\xi(t)), i=1, 2, 3,
\end{equation}
where $f_1(\bx,j):= r-x_1-b_1 x_2-c_1(j) x_3, f_2(\bx,j)= r-x_2-b_2 x_1-c_2(j) x_3, f_3(\bx,j)=e_1(j)x_1+e_2(j)x_2-d$. We assume that $\xi(t)$ is an irreducible continuous time Markov chain that switches from state 1 to 2 at rate  $q_{12}$ and from state 2 to 1 at rate $q_{21}$:
\begin{equation}\label{e:tran}
 \PP\{\xi(t+\Delta)=j~|~\xi(t)=i, \xi(s), s\leq t\}=q_{ij}\Delta+o(\Delta) \text{ if } i\ne j.
\end{equation}
In this setting, the process spends an exponential random time, whose rate can be determined as a function of $q_{12}, q_{21}$, in one environment, after which it switches to the other environment, spends an exponential time there, then switches, and so on. Since $\xi(t)$ is an irreducible Markov chain it will have a unique invariant distribution  on $\M$ given by $$\pi=(\pi_1,\pi_2) =\left(\frac{q_{21}}{q_{12}+q_{21}},\frac{q_{12}}{q_{12}+q_{21}} \right).$$

\subsection {Mathematical setup.} It is well-known that a process $(\BX(t), \xi(t))$ satisfying \eqref{2prey_pdmp} and \eqref{e:tran} is a Markov process with generator acting on functions $G:\R_+^3\times\M\mapsto \R_+^3$ that are continuously differentiable in $\bx$ for each $k\in\M$ as
\begin{equation}
\Lom G(\bx, k)=\sum_{i=1}^3 x_if_i(\bx,k)\frac{\partial G}{\partial x_i}(\bx,k)+\sum_{l\in\M}q_{kl}G(\bx,l).
\end{equation}

We use the norm $\|\bx\|=\sum_{i=1}^3 |x_i|$ in $\R^3$. For $a,b\in\R$, let $a\wedge b:=\min\{a,b\}$ and $a\vee b:=\max\{a,b\}$. Similarly we let $\bigwedge_{i=1}^3 u_i:=\min_{i}u_i$ and  $\bigvee_{i=1}^3 u_i:=\max_{i}u_i$.

The quantity $\PP_{\bx,k}(A)$ will denote the probability of event $A$ if $(\BX(0), \xi(0))=(\bx,k)$. Call $\mu$ an invariant measure for the process $\BX$ if $\mu(\cdot,\cdot)$ is a measure such that for any $k\in \M$ one has that $\mu(\cdot,k)$ is a Borel probability measure on $\R_+^3$ and, if one starts the process with initial conditions distributed according to $\mu(\cdot,\cdot)$, then for any time $t\geq 0$ the distribution of $(\BX(t), \xi(t))$ is given by $\mu(\cdot,\cdot)$.

Let $\Conv\cM$ denote the set of invariant measures of $(\BX(t),\xi(t))$ whose support is contained in $\partial\R^3_+\times\M$. The set of extreme points of $\Conv\cM$, denoted by $\cM$, is the set of ergodic invariant measures with support on the boundary $\partial\R^3_+\times\M$.

We next define what we mean by persistence in our setting.
\begin{deff}
The process $\BX$ is strongly stochastically persistent if it has a unique invariant probability measure $\pi^*$ on $\R^{3,\circ}_+\times\M$ and
\begin{equation}
\lim\limits_{t\to\infty} \|P_\BX(t, \mathbf{x},k, \cdot)-\pi^*(\cdot)\|_{\text{TV}}=0, \;\mathbf{x}\in\R^{3,\circ}_+, k\in \cN
\end{equation}
where $\|\cdot,\cdot\|_{\text{TV}}$ is the total variation norm and $P_\BX(t,\mathbf{x}, k,\cdot)$ is the transition probability of $(\BX(t), \xi(t))$.
\end{deff}

If $\mu\in\M$ is an invariant measure and $\BX$ spends a lot of time close to its support, $\supp(\mu)$, then it will get attracted or repelled in the $i$th direction according to the \textit{Lyapunov exponent}, or invasion rate,
\begin{equation}\label{e:lambda}
\lambda_i(\mu)=\sum_{k\in\cN}\int_{\partial\R^n_+}f_i(\bx, k)\mu(d\bx,k).
\end{equation}
 The intuition comes from noting that $\frac{\ln X_i(t)}{t} = \frac{\ln X_i(0)}{t}+ \frac{\int_0^t f_i(\BX(s), \xi(s))\,ds}{t}$ is approximated well by $\lambda_i(\mu)$ if $t$ is large and $\BX$ stays close to the support of $\mu$.

Piecewise deterministic Markov processes can be quite degenerate and proving that there exist unique invariant probability measures in certain subspaces is far from trivial - see \cite{B18}.

\section{Well-posedness and solutions on the boundary}
In this section we prove some preliminary results which will be useful later on.

\begin{thm}\label{existence_thm}
	For any $(\bx_0,j_0)\in\R^3_+\times\M$ there exists a unique solution $(\BX_t,\xi_t)_{t\geq 0}$ to \eqref{2prey_pdmp} with initial value $(\BX(0),\xi(0))=(\bx_0,j_0)$.
%that remains forever in a compact set with probability 1%
There exists a compact set $\K\subset\R^3_+$ such that every nonnegative solution of \eqref{2prey_pdmp} eventually enters $K$ and then remains there forever.
	Moreover, if $\BX(0)=\bx_0\in \R_+^{3,\circ}$ then with probability one $\BX(t)\in \R_+^{3,\circ}$ for all $t\geq0$.
\end{thm}
\begin{proof}
	
	Because the coefficients of \eqref{2prey_pdmp} are locally Lispchitz, for each initial value, there exists uniquely a local solution to \eqref{2prey_pdmp} (up to a possible explosion time).
	If the initial value is positive, it is clear that the solution will remain positive up to the exposion time because we can write
	\begin{equation}
		\begin{aligned}
			X_1(t)&= e^{\int_0^t (r-X_1(s)-b_1X_2(s)-c_1(\xi_s)X_3(s))ds}\\
			X_2(t)&= e^{\int_0^t (r-X_2(s)-b_2X_1(s)-c_2(\xi_s)X_3(s))ds}\\
			X_3(t)&= e^{\int_0^t (e_1(\xi_s)X_1(s)+e_2(\xi_s)X_2(s)-d)ds}.
		\end{aligned}
	\end{equation}
	On the other hand, it is clear that any solution with nonnegative initial value cannot blow up in a finite time.
	Since
\[
\frac{dX_1}{dt}(t)\leq X_1(t)(r-X_1(s)),
\]
it is clear that if $X_1(0)\geq 0$ then $X_1(t)$ is finite for any $t$. Moreover, eventually, we have $X_1(t)\leq r$. The same conclusion holds true for $X_2(t)$.
	
Note that
\[
\frac{dX_3}{dt}(t)=X_3(t)[e_1(\xi_t)X_1(t)+e_2(\xi_t)X_2(t)-d].
\]
Since we already have shown that $X_1(t), X_2(t)$ are bounded, it is clear from the above that $X_3(t)$ is finite for all $t$.
	
Finally, take $\hat\eps>0$ be sufficiently small such that for all  $i\in\M$ we have
		\begin{equation*}
		\begin{aligned}
			c_1(i)-e_1(i)\hat\eps &\geq 0\\
            c_2(i)-e_2(i)\hat\eps &\leq 0.
		\end{aligned}
	\end{equation*}
	
	From \eqref{2prey_pdmp}, we have for $W_t:=X_1(t)+X_2(t)+\hat\eps X_3(t)$ that
	\begin{align*}
		\frac{dW_t}{dt}\leq& r (X_1(t)+X_2(t)) - X_1(t)^2-X_2(t)^2 -d\hat\eps X_3(t)\\
		\leq &
		(r+d\hat\eps )(X_1(t)+X_2(t))-(X_1(t)^2+X_2(t)^2)- d\hat\eps W_t\\
		\leq & \hat R-d\hat\eps W_t\
	\end{align*}
	for some $\hat R>0$.
	From this equation, it is easy to show that, eventually, we have
	$W_t\leq \frac{\hat R}{d\hat\eps}$
	and if $W_0 \leq \frac{\hat R}{d\hat\eps}$, then $W_t\leq \frac{\hat R}{d\hat\eps}, t\geq0$.
As a result
\[
\left\{(x_1,x_2,x_3)\in\R^{3}_+: x_1+x_2+\hat\eps x_3\leq \frac{\hat R}{d\hat\eps}\right\}
\]
is an attractive invariant set for \eqref{2prey_pdmp}.
\end{proof}
The next assumption is enforced throughout the paper.
\begin{asp}\label{asp0}
The following conditions hold:
\begin{enumerate}
	\item $b_1<1, b_2<1$.
	\item $r\sum_{j\in\M}e_i(j)\pi_j>d; i=1,2.$
	\item $c_1(i)e_1(j) - c_1(j)e_1(i)\ne 0$ for some $i,j\in\M$.
		\item $c_2(i)e_2(j) - c_2(j)e_2(i)\ne 0$ for some $i,j\in\M$.
\end{enumerate}
\end{asp}
Let $\mu_1=\bdelta_{(r,0,0)}\times\M$ and $\mu_2=\bdelta_{(0,r,0)}\times\M$
where $\bdelta_{\bx}$ is the Dirac measure with mass at $\bx$.

Because $(r,0,0)$ and $(0,r,0)$ are equilibria on the axes $Ox_1$ and $Ox_2$ respectively, Assumption \ref{asp0} (2) implies that
$$
\lambda_3(\mu_i)=r\sum_{j\in\M}e_i(j)\pi_j-d>0; i=1,2.
$$
Then in view of \cite{B18} or \cite{du2014asymptotic}, there exists an invariant measure $\mu_{13}$ on $\R^{13,\circ}_+\times\M$ where $\R^{13,\circ}_+:=\{x_1>0, x_3>0, x_2=0\}$ (species $X_2$ is extinct in this subspace) and an invariant measure $\mu_{23}$ on $\R^{23,\circ}_+\times\M$ where $\R^{23,\circ}_+:=\{x_2>0, x_3>0, x_1=0\}$ (species $X_1$ is extinct in this subspace).

On $\R^{12,\circ}_+\times\M$, because $b_1<1,b_2<1$, the point $(x_-,y_-):=\left(\frac{r(1-b_1)}{1-b_1b_2},\frac{r(1-b_2)}{1-b_1b_2}\right)$ will be a saddle equilibrium for the deterministic system
\begin{equation}\label{2comp}
\begin{aligned}
\frac{d X_1}{dt}(t)&=X_1(t)[r-X_1(t)-b_1 X_2(t)]\\
	\frac{d X_2}{dt}(t)&=X_2(t)[r-X_2(t)-b_2 X_1(t)].
\end{aligned}
\end{equation}
Since the coefficients $r, b_1, b_2$ are not influenced by the random switching, the process $\BX$ is fully degenerate and deterministic on $\R^{12,\circ}_+$. As a result, if we let $\bdelta_{x_-,y-}$ be the Dirac measure at $(x_-,y_-)$, then
\begin{equation}\label{e:mu12}
\mu_{12}:=\bdelta_{x_-,y-}\times\pi
\end{equation}
 is the unique invariant probability measure of the process $(\BX,\xi)$ from \eqref{2prey_pdmp} on $\R^{12,\circ}_+\times \M$.

To proceed, we recall the concept of the bracket condition which is an analogue of H\"{o}rmander's condition for hypoelliptic diffusion operators \cite{bakhtin2012invariant}.
Let $[F, G]$ be the Lie bracket of two vector fields $F$ and $G$ and ${\mathcal{F}}_0$ the set of vector fields $\{F_\ell:\ell\in{\M}\}$.
For $k=1,2,...,$ define ${\mathcal{F}}_k={\mathcal{F}}_{k-1}\cup \{[F_\ell,V]: \ell\in{\M}, V\in {\mathcal{F}}_{k-1}\}$, where ${\mathcal{F}}_k(x_1,y,z)$ is the vector space spanned by $\{V(x_1,x_2,x_3):V\in{\mathcal{F}}_k\}$.  Similarly, let ${\mathcal{G}}_0$ = $\{F_\ell- F_m: \ell \neq m\in{\M}\}$ and ${\mathcal{G}}_k={\mathcal{G}}_{k-1}\cup \{[F_\ell,V]: \ell\in{\M}, V\in {\mathcal{G}}_{k-1} \}$. Then we say the weak (strong) bracket condition
is satisfied at  $(x_1,x_2,x_3) \in \R^3$ if there exists $k\geq 0$ such that ${\mathcal{F}}_k(x_1,x_2,x_3) = {\mathbb{R}}^3$ (${\mathcal{G}}_k(x_1,x_2,x_3) = {\mathbb{R}}^3$).

If the weak bracket condition is satisfied at some $x\in\Gamma(K)$, then by \cite[Theorem 1]{bakhtin2012invariant} the invariant measure of the semigroup $(P^{\langle t\rangle})$ is unique and absolutely continuous with respect to the product of the Lebesgue measure on $K$ and the discrete measure on $\M.$
If the strong bracket condition is satisfied at some $x\in\Gamma(K)$, then by  \cite[Theorem 4.5]{benaim2015qualitative} the Markov Chain $(\textbf{Y}_n)$ is irreducible and aperiodic, and every compact subset of $\R^3_+\times\M$ is petite.

We can easily have similar conditions for the absolute continuity of an invariant measure in a subspace of $\R^3_+$.

We next prove that the invariant measures of $(X_1, X_3)$, and mutatis mutandis for $(X_2, X_3)$, are unique.
\begin{thm} \label{t1}
There exist unique invariant measures $\mu_{13}$ and $\mu_{23}$ on $\R^{13,\circ}_+$ and $\R^{23,\circ}_+$ respectively.
\end{thm}

\begin{proof}
From Assumption \ref{asp0} (iii), we can assume without loss of generation that $c_1(1)e_1(2)-c_1(2)e_1(1)\ne0$. We show that $(X_1, X_3)$ satisfies the strong bracket condition (see \cite{benaim2015qualitative}).
Consider the vector fields
\[
G_0 =	F_1 -F_2 =
x_1x_3\left[
\begin{matrix}
	-A \\
	B
\end{matrix}
\right]
\]

and

\[
G_1 =	[F_1, G_0] =
x_1x_3\left[
\begin{matrix}
	Ad+(A-C)x_1\\
	Cx_3-Bx_1+Br
\end{matrix}
\right],
\]
where
$A=c_1(1) - c_1(2)$, $B= e_1(1) - e_1(2)$ and $C=c_1(1)e_1(2) - c_1(2)e_1(1).$
Then \[
G_2 =	[G_0,G_1] =
x_1x_3\left[
\begin{matrix}
	x_1(A(2C-A)x_3-BCx_1+AB(d+r)) \\
	x_3(B(2C+A)x_1-ACx_3-AB(d+r))
\end{matrix}
\right].
\]

The determinant of the matrix $\left[
\begin{matrix}
	G_0 & G_1
\end{matrix}
\right]$
is given by
\[
	x_1^2x_3^2\left(BCx_1-ABr-ABd-ACx_3\right).
\]
which vanishes when $x_1=\frac{A}{BC}(B(d+r)+Cx_3).$
Next the determinant of $\left[
\begin{matrix}
	G_0 & G_2
\end{matrix}
\right]$ is given by
\[
x_1^2x_3^2(B^2Cx_1^2+A^2Cx_3^2-(AB^2d+AB^2r)x_1+(A^2Bd+A^2Br)x_3-4ABCx_1x_3),
\]
which is zero when $$x_1=
\frac{A}{2BC}(B(d + r) + 4Cx_3 \pm \sqrt{D}),$$ where $$D=B^2(d + r)^2 + 4BC(d + r)x_3 + 12C^2x_3^2.$$
Then the strong bracket condition may be unsatisfied when $$2(B[d+r]+Cx_3)=B(d+r)+4Cx_3\pm \sqrt{D}. $$ This implies $$8Cx_3(B[d+r]+Cx_3)=0, $$ i.e., $x_3=-B(d+r)/C.$
This gives $\frac{A}{BC}(B(d+r)+Cx_3)=0.$
So the strong bracket condition is satisfied in $\R^{23,\circ}$.

Now, by  \cite[Theorem 4.4, Theorem 4.6]{benaim2015qualitative}, the probability measure $\PP_{(\bx_0,j_0)}[(\BX_t,\xi_t)\in\cdot\times\{j_0\}]$ is absolutely continuous with respect to Lebesgue measure on $\R^{13,\circ}_+$, and there exists a unique invariant probability measure $\mu_{13}$ on $\R^{13,\circ}_+\times\M$. In addition there are constants $c>1$ and $\alpha>0$ such that for any $t\geq 0, \bx\in \R^{13,\circ}_+, j\in \M$ we have $\|\PP_{(\bx,j)}[(\BX_t,\xi_t)\in\cdot]-\mu_{13}\|_{TV}\leq c e^{-\alpha t},$ so that the convergence is exponential.

\end{proof}

Now we present some auxiliary lemmas needed to obtain the main results.
\begin{lm}\label{lm3.2}
%	For any invariant probability measure $\pi$ of $\BX$ one has
%	\begin{equation}\label{lm3.2-e1}
%	\int_{\R^3_+}\sum_{j\in \M}\left(\dfrac{\sum c_ix_if_i(\bx)}{1+\bc^\top\bx}-\dfrac{1}2\dfrac{\sum \sigma_{ij}c_ic_jx_ix_jg_i(\bx)g_j(\bx)}{(1+\bc^\top\bx)^2}\right)\pi(d\bx)=0.
%	\end{equation}
%	Furthermore,
	$$\int_{\R^3_+}\sum_{j\in\M}\dfrac{x_1f_1(\bx,j)+x_2f_2(\bx,j)}{x_1+x_2}\pi(d\bx,j)=0, \pi\in\{\mu_1,\mu_2,\mu_{13},\mu_{23}\}.$$
\end{lm}
\begin{rem}
	Note that even though $\frac{1}{x_1+x_2}$ is undefined on the set $E_0:=\{(x_1,x_2,x_3)\in \R_+^3~|~x_1+x_2=0\}$ this does not matter since none of the measures $\{\mu_1,\mu_2,\mu_{13},\mu_{23}\}$ put any mass on the set $E_0$.
\end{rem}
\begin{proof}
%	We show in \cite[Lemma 3.3]{HN16} that
%	\begin{equation}\label{lm3.2-e1}
%	\int_{\R^3_+}\left(\dfrac{\sum c_ix_if_i(\bx)}{1+\bc^\top\bx}-\dfrac{1}2\dfrac{\sum \sigma_{ij}c_ic_jx_ix_jg_i(\bx)g_j(\bx)}{(1+\bc^\top\bx)^2}\right)\pi(d\bx)=0
%	\end{equation}
%	for any invariant probability measure $\pi$.
	To prove the  lemma, one can use a contradiction argument similar to \cite[Lemma 3.3 and Lemma 5.1]{HN16}.
\end{proof}
\begin{lm}\label{l:lyapunov}
	For any ergodic measure $\mu\in\cM$ we have that $\lambda_i(\mu)$ is well defined and finite. Furthermore,
	\[
	\lambda_i(\mu)=0,~i\in I_\mu.
	\]
\end{lm}
\begin{proof}
	The proof is the same as the proof of \cite{HN16}[Lemma 5.1].
\end{proof}

%We start by proving some general results due to \eqref{a.tight}. In view of \eqref{a.tight},
%there is $M>0$
%such that
%\begin{equation}\label{e2.5}
%\left[\dfrac{\sum c_ix_if_i(\bx)}{1+\sum c_ix_i}-\dfrac12\dfrac{\sum \sigma_{ij}c_ic_jx_ix_jg_i(\bx)g_j(\bx)}{(1+\sum c_ix_i)^2}+\gamma_b\left(1+\sum_{i=1}^n (|f_i(\bx)|+g_i^2(\bx))\right)\right]<0
%\end{equation}
%if $\|\bx\|\geq M$.
%Since
%\[
%|g_i(\bx)g_j(\bx)\sigma_{ij}|\leq 2|\sigma_{ij}|(|g_i(\bx)|^2+|g_j(\bx)|^2)
%\]
%we can find $\delta_0\in(0,0.5\gamma_b)$ such that
%\begin{equation}\label{e2.6}
%3\delta_0\sum |g_i(\bx)g_j(\bx)\sigma_{ij}|+\delta_0\sum g_i^2(\bx)\leq \gamma_b\sum g_i^2(\bx)\,,\,\bx\in\R^n_+.
%\end{equation}
%In view of \eqref{e2.5} and \eqref{e2.6}, we have
%\begin{equation}\label{e2.0}
%\begin{aligned}
%\dfrac{\sum c_ix_if_i(\bx)}{1+\bc^\top\bx}&-\dfrac{1}2\dfrac{\sum \sigma_{ij}c_ic_jx_ix_jg_i(\bx)g_j(\bx)}{(1+\sum c_ix_i)^2}+\gamma_b+\delta_0\sum(2|f_i(\bx)|+g^2_i(\bx))\\
%&+3\delta_0\sum |g_i(\bx)g_j(\bx)\sigma_{ij}|<0\,\text{ for all }\, \|\bx\|\geq M.
%\end{aligned}
%\end{equation}
%Using \eqref{e2.0} one can define
%\begin{equation}\label{e:H}
%\begin{aligned}
%H:=\sup\limits_{\bx\in\R^{3}_+}\Bigg\{&\dfrac{\sum c_ix_if_i(\bx)}{1+\bc^\top\bx}-\dfrac{1}2\dfrac{\sum \sigma_{ij}c_ic_jx_ix_jg_i(\bx)g_j(\bx)}{(1+\sum c_ix_i)^2}\\
%&+\gamma_b+\delta_0\sum(2|f_i(\bx)|+g_i^2(\bx))
%+3\delta_0\sum |g_i(\bx)g_j(\bx)\sigma_{ij}|\Bigg\}<\infty.
%\end{aligned}
%\end{equation}
Define the normalized occupation measures $\Pi^{\bx,j}_{t}$ by
\begin{equation}\label{e:occupation}
	\Pi^{\bx,j}_{t}(d\by,i):=\frac{1}{t}\int_0^t\PP_{\bx,j}\{\BX(s)\in d\by,\xi(s)=i\}\,ds.
\end{equation}
and the random normalized occupation measures by
\begin{equation}\label{e:occupation_rand}
	\widetilde\Pi_{t}(d\by,i):=\frac{1}{t}\int_0^t\1_{\{\BX(s)\in d\by,\xi(s)=i\}}\,ds.
\end{equation}

\begin{lm}\label{lm3.3}
	Suppose the following:

	\begin{itemize}
		\item  The sequences $\{(\bx_k,j_k)\}_{k\in N}\subset %\R_+^3
		\K\times\M, (T_k)_{k\in \N}\subset \R_+$ are such that %$\|\bx_k\|\leq M$,
		$T_k>1$ for all $k\in \N$ and $\lim_{k\to\infty}T_k=\infty$.
		
		\item The sequence $(\Pi^{\bx_k, j_k}_{T_k})_{k\in \N}$ converges weakly to an invariant probability measure
		$\pi$.
		
%		\item The function $h:\R^3_+\to\R$ is any upper semi-continuous function satisfying
%		$|h(\bx)|<K_h(1+\bc^\top\bx)^\delta(1+\sum_i|f_i(\bx)|$, $\bx\in \R^n_+$,
%		for some $K_h\geq 0$, $\delta<\delta_0$.
	\end{itemize}
	Then for any function $h(\bx,i):\K\times\M\to\R$ that is upper semi-continuous (in $x$ for each fixed $i$), one has
	\begin{equation}\label{lm3.3-e1}
	\lim_{k\to\infty}\int_{\R^3_+}\sum_{j\in\M} h(\bx,j)\Pi^{\bx_k,j_k}_{T_k}(d\bx,j)\leq \int_{\R^3_+}\sum_{j\in\M}h(\bx,j)\pi(d\bx,j).
	\end{equation}
\end{lm}
\begin{proof}
%	If the function $h$ is bounded and upper continuous,
Since $\K$ is a compact set, \eqref{lm3.3-e1} can be obtained directly from the Portmanteau theorem.
The details are left to the readers.
%	In case $h$ satisfies $|h(\bx)|<K_h(1+\bc^\top\bx)^\delta(1+\sum_i(|f_i(\bx)|+|g_i(\bx)|^2))$, $\bx\in \R^n_+$,
%	for some $K_h\geq 0$, $\delta<\delta_0$,
%	we use the uniform bound in \cite[Lemma 3.3]{HN16} and the truncated arguments
%	in \cite[Lemma 3.4]{HN16} to obtain \eqref{lm3.3-e1}.
%	The details are omitted here.
\end{proof}
\section{Persistence}
For $\mu\in\cM$ remember that the invasion rate of species $i$ with respect to $\mu\in\M$ is defined by \[ \lambda_i(\mu) = \sum_{j\in\M}\int_{\partial\R^3_+}f_i(\bx,j)\mu(d\bx,j).\]
We assume that
\begin{equation}\label{persistence_assumption}
	\lambda_2(\mu_{13})>0 \text{ and } \lambda_1(\mu_{21})>0.\end{equation}
Using \eqref{e:mu12} we have $$\lambda_{3}(\mu_{12})=\sum_{j\in\M} (e_1(j)x_-+e_2(j)y_-d)\pi_j=\bar e_1\frac{r(1-b_1)}{1-b_1b_2}+\bar e_2\frac{r(1-b_2)}{1-b_1b_2}-d.$$

Since $\bar e_1 r>d; \bar e_2 r>d$ and $\frac{1-b_1}{1-b_1b_2}+\frac{1-b_1}{1-b_1b_2}>1$, (which can be easily checked using $b_1<1, b_2<1$, we have
\begin{equation}\label{lambda3mu12}
	\lambda_{3}(\mu_{12})>0.
	\end{equation}

By the minimax principle, \eqref{persistence_assumption} and \eqref{lambda3mu12} are equivalent to the existence of $p_1, p_2, p_3>0$ satisfying
\begin{equation}\label{c1e1}
\sum_{i=1}^3p_i\lambda_i(\pi)>0, \pi\in\{\mu_1,\mu_2,\mu_{13},\mu_{23},\mu_{12}\}.
\end{equation}
Let $p_0$ be sufficiently large (compared to $p_1,p_2,p_3$) such that
\begin{equation}\label{e2.3}
p_0\min\{\lambda_1(\bdelta^*),\lambda_2(\bdelta^*)\}+\sum_{i=1}^3p_i\lambda_i(\bdelta^*)>0.
\end{equation}

%By rescaling $p_0,\dots,p_3$, we can assume that $\sum_{i=0}^3 p_i\leq\dfrac{\delta_0}4.$
Define
\begin{equation}\label{e2.3}
2\rho^*:=\min\left\{p_0\min\{\lambda_1(\bdelta^*),\lambda_2(\bdelta^*)\}+\sum_{i=1}^3p_i\lambda_i(\bdelta^*), \sum_{i=1}^3p_i\lambda_i(\pi), \pi\in\{\mu_1,\mu_2,\mu_{13},\mu_{23},\mu_{12}\}\right\}>0.
\end{equation}
%and $P_{\delta}=\left\{\hat \bp:=(\hat p_0,\cdots,\hat p_3)\in\R^4: |\hat p_0|+|\hat p_1|+|\hat p_2|+|\hat p_3|\leq\dfrac{\delta_0}4\right\}$.
Let $\K$ be the attractive compact mentioned in Theorem \ref{existence_thm} and $\K^\circ=\R^{3,\circ}_+\cap \K$.

%For any $\hat \bp$ define the function $V_{\hat\bp}: \K{\circ}_+\to\R_+$ by
%\begin{equation}\label{e:V}
%V_{\hat \bp}(\bx)=\dfrac{1}{(x_1+x_2)^{\hat p_0}\prod_{i=1}^3 x_i^{\hat p_i}}.
%\end{equation}
%Note that if
%\[
%Z:=\ln  V_{\hat \bp} =-\hat p_0 \ln(x_1+x_2) -\sum_{i=1}^3 \hat p_i\ln x_i
%\]
%then we can write $V_{\hat \bp}^{\delta_0} = e^{\delta_0Z}$. Taking derivatives yields $$\frac{\partial\left( V_{\hat \bp}^{\delta_0}\right)}{\partial x_i}=\delta_0V_{\hat \bp} \frac{\partial Z}{\partial x_i}$$ and
%\[
%\frac{\partial^2\left( V_{\hat \bp}^{\delta_0}\right)}{\partial x_i\partial x_j}=\delta_0V_{\hat \bp} \left(\delta_0 \frac{\partial Z}{\partial x_i} \frac{\partial Z}{\partial x_j} +  \frac{\partial^2 Z}{\partial x_i \partial x_j}\right).
%\]
%Using these expressions and the definition of the generator $\Lom$ one can show, after some computations, that
%\begin{equation}\label{lv_delta}
%\begin{aligned}
%\Lom V_{\hat \bp}^{\delta_0}(\bx,j)=\delta_0V_{\hat \bp}^{\delta_0}(\bx)\bigg[&-\hat p_1f_1(\bx,j)-\hat p_2f_2(\bx,j)-\hat p_3f_3(\bx,j)\\
%&-\hat p_0\dfrac{x_1f_1(\bx,j)+x_2f_2(\bx,j)}{x_1+x_2}\bigg].
%\end{aligned}
%\end{equation}
%In virtue of \eqref{e2.0}, we have
%\begin{equation}\label{e2.8}
%\Lom V_{\hat \bp}^{\delta_0}(\bx)<-\gamma_b\delta_0 V^{ \delta_0}_{\hat\bp}(\bx) \text{ for } \bx\in\R_+^{3,\circ},  \|\bx\|>M, \hat\bp\in P_{\delta_0}.
%\end{equation}
%Analogously, using \eqref{e:H}
%\begin{equation}\label{e2.9}
%\Lom V_{\hat\bp}^{ \delta_0}(\bx)<H\delta_0 V^{\delta_0}_{\hat\bp}(\bx),\bx\in\R^{3,\circ}_+, \hat\bp\in P_{\delta_0}.
%\end{equation}

Let $\bp=(p_0,\cdots,p_3)$ satisfy \eqref{e2.3} and
consider the function
\begin{equation}\label{e:V}
V(\bx):=V_{ \bp}(\bx)=\dfrac{1}{(x_1+x_2)^{ p_0}\prod_{i=1}^3 x_i^{ p_i}}.
\end{equation}
%Let $y_i=\dfrac{x_i}{x_1+x_2}, i=1,2$. Since $y_1+y_2=1$ we have the following estimate
%$$
%\begin{aligned}
%&\dfrac{x_1f_1(\bx)+x_2f_2(\bx)}{x_1+x_2}-\dfrac{\sum_{i,j=1}^2\sigma_{ij}x_ix_jg_i(\bx)g_j(\bx)}{2(x_1+x_2)^2}\\
%&\qquad=\sum_{i=1}^2y_i\left(f_i(\bx)-\dfrac{g_i^2(\bx)\sigma_{ii}}2\right) -y_1y_2\sigma_{12}g_1(\bx)g_2(\bx) \\
%&\qquad=\sum_{i=1}^2y_i\left(f_i(\bx)-\dfrac{g_i^2(\bx)\sigma_{ii}}2\right)+\dfrac12(y_1+y_2)\left(y_1g_1^2(\bx)\sigma_{11}+y_2g_2^2(\bx)\sigma_{22}\right)\\
%&\qquad\qquad+\dfrac12\left(-y_1^2g_1^2(\bx)\sigma_{11}-y_2^2g_2^2(\bx)\sigma_{11}-2y_1y_2\sigma_{12}g_1(\bx)g_2(\bx)\right)\\
%&\qquad=y_1\left(f_1(\bx)-\dfrac{g_1^2}2(\bx)\right) + y_2\left(f_2(\bx)-\dfrac{g_2^2}2(\bx)\right)
%+\dfrac12y_1y_2\left(g_1^2(\bx)\sigma_{11}+g_2^2(\bx)\sigma_{22}-2g_1(\bx)g_2(\bx)\sigma_{12}\right).
%\end{aligned}
%$$
%Since $(g_i(\bx)g_j(\bx)\sigma_{ij})_{3\times 3}$ is positive definite it is clear that $$g_1^2(\bx)\sigma_{11}+g_2^2(\bx)\sigma_{22}-2\sigma_{12}g_1(\bx)g_2(\bx)\geq 0.$$
It is readily seen that
\begin{equation}\label{e2.4}
\dfrac{x_1f_1(\bx)+x_2f_2(\bx)}{x_1+x_2}
\geq \min\left\{f_1(\bx), f_2(\bx)\right\}.
\end{equation}
Define $\Phi:\{\R_+^{3}\setminus\{(x_1,x_2,x_3)\in\R_+^3~|~x_1+x_2=0\}\}\times \M\mapsto\R$ by
\begin{equation*}
\begin{aligned}
\Phi(\bx,j)=&-p_1f_1(\bx,j)-p_2f_2(\bx,j)-p_3f_3(\bx,j)\\
&-p_0\frac{x_1f_1(\bx,j)+x_2f_2(\bx,j)}{x_1+x_2}.
\end{aligned}
\end{equation*}
Let $\hat\Phi:\R^3_+\times \M\mapsto\R$ be the function
\begin{equation}\label{ehatphi}
\begin{aligned}
\hat\Phi(\bx,j)=&-p_1f_1(\bx,j)-p_2f_2(\bx,j)-p_3f_3(\bx,j)\\
&-p_0\min\left\{f_1(\bx,j), f_2(\bx,j)\right\}.
\end{aligned}
\end{equation}
Define $\widetilde\Phi:\R^3_+\times\M\mapsto\R$ by
$$
\widetilde\Phi(\bx,j)=
\begin{cases}
&\hat \Phi(\bx,j),  \text{ if } ~~~x_1+x_2=0.\\
&\Phi(\bx,j),  \text{ if } ~~~x_1+x_2\ne 0.
\end{cases}
$$
In view of \eqref{e2.4}, for each $j\in \M$, $\widetilde\Phi(\bx,j)$ is an upper semi-continuous function.

%Let $n^*\in\N$ such that
%\begin{equation}\label{e:n*}
%\gamma_b(n^*-1)>H.
%\end{equation}
\begin{lm}\label{lm3.1}
	Suppose that \eqref{persistence_assumption} holds. Let $\bp$ and $\rho^*$ be as in \eqref{e2.3}.
	There exists a $T>0$ such that for any  $\bx\in\partial\R^3_+\cap\K,  j\in\M$ one has
	\begin{equation}\label{lm3.1-e1}
	\dfrac1T\int_0^T\E_{\bx,j}\widetilde\Phi(\BX(t),\xi(t))dt\leq-\rho^*.
	\end{equation}
	As a corollary,
	there is a $\tilde\delta>0$ such that
	\begin{equation}\label{lm3.1-e2}\dfrac1T\int_0^T\E_{\bx,j}\Phi(\BX(t),\xi(t))dt\leq-\dfrac34\rho^*,
	\end{equation}
	for any $(\bx,j)\in\K\times \M$ satisfying  dist$(\bx,\partial\R^3_+)<\tilde\delta.$
\end{lm}
\begin{proof}
	We argue by contradiction to obtain \eqref{lm3.1-e1}. Suppose that the conclusion of this lemma is not true.
	Then, we can find $(\bx_k,j_k)\in\partial\R^3_+\times\M, \|\bx_k\|\leq M$
	and $T_k>0$, $\lim_{k\to\infty} T_k=\infty$
	such that
	\begin{equation}\label{e3.9}
	\dfrac1{T_k}\int_0^{T_k}\E_{\bx_k}\widetilde\Phi(\BX(t),\xi(t))dt>-\rho^*\,,\,k\in\N.
	\end{equation}
	Remember that the normalized occupation measures are defined by
	\[
	\Pi^{\bx,j}_{t}(d\by,i):=\frac{1}{t}\int_0^t\PP_{\bx,j}\{\BX(s)\in d\by,\xi(s)=i\}\,ds.
	\]
	It follows from \cite[Lemma 4.1]{HN16} that $\left(\Pi^{\bx_k,j_k}_{T_k}\right)_{k\in \N}$ is tight. As a result
	$\left(\Pi^{\bx_k,j_k}_{T_k}\right)_{k\in \N}$ has a convergent subsequence in the weak$^*$-topology.
	Without loss of generality, we can suppose that $\left(\Pi^{\bx_k,j_k}_{T_k}\right)_{k\in \N}$
	is a convergent sequence in the weak$^*$-topology.
	It can be shown (see Lemma 4.1 from \cite{HN16} or Theorem 9.9 from \cite{EK09}) that its limit is an invariant probability measure $\mu$ of $(\BX,\xi)$. Since $(\bx_k,j_k)\in\partial\R^3_+\times\M$, the support of $\mu$ lies in $\partial \R_+^3\times\M$.
	As a consequence of Lemma \ref{lm3.3}
	$$\lim_{k\to\infty}\dfrac1{T_k}\int_0^{T_k}\E_{\bx_k,j_k}\widetilde\Phi(\BX(t),\xi(t))dt\leq\int_{\R^3_+}\sum_{j\in\M}\widetilde\Phi(\bx,j)\mu(d\bx,j).$$
	Using Lemmas \ref{lm3.2} and \ref{l:lyapunov},  together with equation \eqref{e2.3} we get that
	$$\lim_{k\to\infty}\dfrac1{T_k}\int_0^{T_k}\E_{\bx_k,j_k}\widetilde\Phi(\BX(t),\xi(t))dt\leq -2\rho^*.$$
	This contradicts \eqref{e3.9}, which means \eqref{lm3.1-e1} is proved.
	
	With $\hat\Phi$ defined in \eqref{ehatphi}, we have $\hat\Phi(\bx,j)\geq\Phi(\bx,j)$ for $x_1+x_2\ne0$ and $\hat\Phi(\bx,j)=\widetilde\Phi(\bx,j)$ if $x_1+x_2=0$. As a result of \eqref{e2.3}
	$$
	\hat\Phi(\0)=\widetilde\Phi(\0)=-\sum\left(p_if_i(\0)\right)-p_0\min\left\{f_1(\0), f_2(\0)\right\}\leq-2\rho^*.
	$$
	Thus
	\begin{equation}
	\dfrac1T\int_0^T\E_{(0,0,x_3),j}\hat\Phi(\BX(t),\xi(t))dt=
	\dfrac1T\int_{0}^T\E_{(0,0,x_3),j}
	\widetilde\Phi(\BX(t),\xi(t))dt\leq-\rho^*, (0,0,x_3)\in\K.
	\end{equation}
	Due to the Feller property of $(\BX(t),\xi(t))$ on $\R^3_+\times\M$ and the continunity of $\hat\Phi$ on $\R^3_+$, there is an $\hat\eps>0$ such that
	\begin{equation}\label{lm3.1-e5}
	\dfrac1 T\int_0^T\E_{\bx,j}\hat\Phi(\BX(t),\xi(t))dt\leq -\frac34\rho^*,\text{ if } x_1+x_2\leq\hat\eps, (\bx,j)\in\K\times\M .
	\end{equation}
	Together with $\Phi(\bx,j)\leq\hat\Phi(\bx,j), x_1+x_2\ne0,$ this implies
	$$
	\dfrac1T\int_0^T\E_{\bx,j}\Phi(\BX(t),\xi(t))dt\leq -\frac34\rho^*,\,\, (\bx,j)\in\K\times \M,x_1+x_2\leq\hat\eps.
	$$
	If $x_1+x_2\ne 0$, then $$\PP_{\bx,j}\left\{\widetilde\Phi(\BX(t),\xi(t))=\Phi(\BX(t),\xi(t)), t\geq 0 \right\}=1.$$
	Using the Feller property of $(\BX(t))$ on $\{(x_1,x_2,x_3)\in\R_+^3~|~x_1+x_2\ne 0\}$, equation \eqref{lm3.1-e1} and the continuity of $\Phi(t)=\wtd\Phi(t)$ on $\{(x_1,x_2,x_3)\in\R_+^3~|~x_1+x_2\ne 0\}$ one can see that there exists $\tilde\delta\in(0,\hat\eps)$ for which
	\begin{equation}\label{lm3.1-e6}
	\dfrac1T\int_0^T\E_{\bx,j}\Phi(\BX(t),\xi(t))dt\leq -\frac34\rho^*, (\bx,j)\in\K^\circ\times\M, x_1+x_2\geq\hat\eps,  \text{dist}(\bx,\partial\R^3_+)<\tilde\delta.
	\end{equation}
	Combining \eqref{lm3.1-e5} and \eqref{lm3.1-e6} yields \eqref{lm3.1-e2}.

\end{proof}
\begin{lm}\label{lm2.5}
	Let $Y$ be a random variable, $\theta_0>0$ a constant, and suppose $$\E \exp(\theta_0 Y)+\E \exp(-\theta_0 Y)\leq K_1.$$
	Then the log-Laplace transform
	$\phi(\theta)=\ln\E\exp(\theta Y)$
	is twice differentiable on $\left[0,\frac{\theta_0}2\right)$ and
	$$\dfrac{d\phi}{d\theta}(0)= \E Y,$$
	$$0\leq \dfrac{d^2\phi}{d\theta^2}(\theta)\leq K_2\,, \theta\in\left[0,\frac{\theta_0}2\right)$$
	for some $K_2>0$ depending only on $K_1$.
\end{lm}
\begin{proof}
	See Lemma 3.5 in \cite{HN16}.
\end{proof}
\begin{prop}\label{prop2.1}
	Let $V$ be defined by \eqref{e:V} with $\bp$ and $\rho^*$ satisfying \eqref{e2.3} and $T>0$ satisfying the assumptions of Lemma \ref{lm3.1}.
	There are $\theta\in\left(0,1\right)$, $K_\theta>0$, such that for  $\bx\in\K^\circ$,
	\begin{equation}\label{c-LV}\E_\bx V^\theta(\BX(T))\leq \exp(-0.5\theta \rho^*T) V^\theta(\bx)+K_\theta.
		\end{equation}
\end{prop}
\begin{proof}
	We have
	\begin{equation}\label{e:G}
	\ln V(\BX(T))=\ln V(\BX(0)) +\int_0^T\Phi(\BX(t),\xi(t))dt.
	\end{equation}

Since $\Phi$ is bounded on $\K\times\M$, we can easily have that
\begin{equation}\label{e3.5}
\exp\{-HT\}\leq \frac{V(\BX(T))}{V(\bx)}\leq \exp\{HT\}, \bx\in\K,
\end{equation}
for some nonrandom constant $H$.
%	where
%	\begin{equation}\label{e:GG}
%	\begin{aligned}
%	G(T)=&\int_0^T\Phi(\BX(t))dt+p_0\int_0^T\dfrac{X_1(t)g_1(\BX(t))dE_1(t)+X_2(t)g_2(\BX(t))dE_2(t)}{X_1(t)+X_2(t)}\\
%	&+ \int_0^T\left[\dfrac{\sum_ic_iX_i(t)g_i(\BX(t))dE_i(t)}{1+\bc^\top\BX(t)}-\sum_ip_ig_i(\BX(t))dE_i(t)\right].
%	\end{aligned}
%	\end{equation}
Thus,	
 the assumptions of Lemma \ref{lm2.5} hold for the random variable $\int_0^T\Phi(\BX(t),\xi(t))dt=\frac{V(\BX(T))}{V(\bx)}\leq \exp\{HT\}$. Therefore,
	there is $\tilde K_2\geq 0$ such that
	\begin{equation}\label{e:phi''}
	0\leq \dfrac{d^2\tilde\phi_{\bx,j,T}}{d\theta^2}(\theta)\leq \tilde K_2\,\text{ for all }\,\theta\in\left[0,1\right),\, (\bx,j)\in\R^{3,\circ}_+\times\M, \|\bx\|\leq M, T\in [T^*,n^*T^*]
	\end{equation}
	where
	$$\tilde\phi_{\bx,j,T}(\theta)=\ln\E_{\bx,j} \exp\Big(\theta\int_0^T\Phi(\BX(t),\xi(t))dt\Big).$$
	
	An application of Lemma \ref{lm3.1}, and equation \eqref{e:G} yields
	\begin{equation}\label{e:phi'}
	\dfrac{d\tilde\phi_{\bx,j,T}}{d\theta}(0)=\E_{\bx,j} \int_0^T\Phi(\BX(t),\xi(t))dt\leq -\dfrac34\rho^*T
	\end{equation}
	for all $(\bx,j)\in\K^\circ_+$ satisfying  dist$(\bx,\partial\R^3_+)<\tilde\delta.$
	By a Taylor expansion around $\theta=0$, for $\bx\in\K, \dist(\bx,\partial\R^n_+)<\tilde\delta,$ and $\theta\in\left[0,1\right)$ and using \eqref{e:phi''}-\eqref{e:phi'} we have
	$$\tilde\phi_{\bx,j,T}(\theta)=\tilde\phi_{\bx,j,T}(0)+\dfrac{d\tilde\phi_{\bx,j,T}}{d\theta}(0)\theta+ \frac{1}{2} \dfrac{d^2\tilde\phi_{\bx,T}}{d\theta^2}(\theta') (\theta-\theta')^2\leq -\dfrac34\rho^*T\theta+\theta^2\tilde K_2 .$$
	If we choose any $\theta\in\left(0,1\right)$ satisfying
	$\theta<\frac{\rho^*T^*}{4\tilde K_2}$, we obtain that
	\begin{equation}\label{e3.10}
	\tilde\phi_{\bx,j,T}(\theta)\leq -\dfrac12\rho^*T\theta\,\,\text{ for all }\,(\bx,j)\in\R^{3,\circ}\times\M,\|\bx\|\leq M, \dist(\bx,\partial\R^n_+)<\tilde\delta, T\in [T^*,n^*T^*],
	\end{equation}
	which leads to
	\begin{equation}\label{e3.11}
	\dfrac{\E_{\bx,j} V^\theta(\BX(T))}{V^\theta(\bx)}=\exp \tilde\phi_{\bx,j,T}(\theta)\leq\exp(-0.5p^*T\theta).
	\end{equation}
	In view of \eqref{e3.5},
	we have for $(\bx,j)\in\K^\circ\times\M$ satisfying $\dist(\bx,\partial\R^n_+)\geq\tilde\delta$ that
	\begin{equation}\label{e3.12}
	\E_{\bx,j} V^\theta(\BX(T))\leq \exp(\theta TH)\sup\limits_{\bx\in\K, \dist(\bx,\partial\R^n_+)\geq\tilde\delta}\{V(\bx)\}=:K_\theta<\infty.
	\end{equation}
	The proof can be finished by combining \eqref{e3.11} and \eqref{e3.12}.
\end{proof}
\begin{thm}\label{mthm-2}
	Suppose
	$$
	\lambda_2(\mu_{13})=\int_{\R^{13,\circ}_+}\sum_{j\in\M} (r- b_2x_1-c_2(j)x_3)\mu_{13}(dx_1,dx_3,j)>0
	$$
	and
	$$
	\lambda_1(\mu_{23})=\int_{\R^{23,\circ}_+}\sum_{j\in\M} (r- b_1x_2-c_1(j)x_3)\mu_{23}(dx_2,dx_3,j)>0
	$$
	where $\mu_{13}$ is the (unique) invariant measure on $\R^{13,\circ}_+\times\M$
	and
	$\mu_{23}$ is the (unique) invariant measure on $\R^{23,\circ}_+\times\M$. Then
	for each $\eps>0$, there exists $\delta>0$ such that
	$$
	\liminf_{t\to\infty}\PP_\bx\{X_i\geq\delta,i=1,2,3\}\geq 1-\eps.
	$$
	If the strong bracket condition is satisfied in $\K^\circ$ then
	the system is strongly stochastically persistent.

\end{thm}

\begin{proof}
In Proposition \ref{prop2.1}, we have constructed a Lyapunov function $V$ satisfying \eqref{c-LV}. This inequality and \eqref{e3.5} show that
$$
\limsup_{t\to\infty}\E_\bx V^{\theta}(\BX(t))\leq K_0
$$
for a nonrandom $K_0$ independent of $\bx$; see e.g. \cite[Theorem 3.3]{watts2021dynamics} or \cite[Theorem 2.2]{tuong2019extinction}.
As a result, because $\lim_{x_1\wedge x_2\wedge x_3\to0}V(\bx)=0$,
for each $\eps>0$, there exists $\delta>0$ such that
$$
\liminf_{t\to\infty}\PP_\bx\{X_i\geq\delta,i=1,2,3\}\geq 1-\eps.
$$

\end{proof}

\section{Extinction}
Piecewise deterministic Markov processes can be quite degenerate, and one has to do some additional work in order to see which parts of the state space are visited by the process. Let $\phi_t^k(\cdot)$ be the flow associated with
the equation
$$dX_i(t)=X_i(t)f_i(\BX(t),k)dt,i=1,\dots,n$$
for each fixed $k\in\cN$.
That is, $\phi_t^k(\bx)$
is the solution at time $t$
to
$$dX_i(t)=X_i(t)f_i(\BX(t),k)dt,i=1,\dots,n$$
with initial value $\BX(0)=\bx$. Define the orbit
$$\gamma^+(\bx)=\left\{\phi_{t_n}^{k_n}\circ\cdots\circ\phi_{t_1}^{k_1}(\bx): n\in \Z_+, t_l\geq 0, k_l\in\M: l=1,\dots,n\right\}
$$
and for the invariant set $K\subset \R^{3}_+$ let
$$\Gamma(K)=\bigcap_{\bx\in K^\circ}\bar{\gamma^+(\bx)}$$
be, the possibly empty, compact subset which is accessible for the process $(\BX(t), \xi(t))$ from $K$.
\begin{thm}\label{t:extinction} We have the following extinction results:
	\begin{enumerate}
		\item
	If $\lambda_2(\mu_{13})<0$ then for any compact set $\K_{13}\subset\R^{13,\circ}$ any for any $\eps>0$, there exists $\delta>0$ such that for all $(x_1,x_3)\in\K_{13}, 0<x_2<\delta$ we have
	$$
	\PP_{\bx,i}\left\{\lim_{t\to\infty} \frac{X_2(t)}{t}=\lambda_2(\mu_{13})<0 \right\}\geq 1-\eps.
	$$
\item 	If $\lambda_1(\mu_{23})<0$ then for any compact set $\K_{23}\subset\R^{23,\circ}$ and for any $\eps>0$, there exists $\delta>0$ such that for all $(x_2,x_3)\in\K_{23}, 0<x_1<\delta$ we have
	$$
	\PP_{\bx,i}\left\{\lim_{t\to\infty} \frac{X_1(t)}{t}=\lambda_1(\mu_{23})<0 \right\}\geq 1-\eps.
	$$
	\end{enumerate}
	\begin{enumerate}
	\item[(3)] If $\lambda_2(\mu_{13})<0, \lambda_1(\mu_{23})>0$ and $\R^{13,\circ}_+$ is accessible from any $\bx\in\R^{3,\circ}_+$, that is $\R^{13,\circ}_+\cap\Gamma(K) \ne\emptyset$, then
$$
\PP_{\bx,i}\left\{\lim_{t\to\infty} \frac{X_2(t)}{t}=\lambda_2(\mu_{13})<0 \right\}=1.
$$
	\item[(4)]
	If $\lambda_1(\mu_{23})<0, \lambda_2(\mu_{13})>0$ and $\R^{23,\circ}_+$ is accessible from any $\bx\in\R^{3,\circ}_+$ then
	$$
	\PP_{\bx,i}\left\{\lim_{t\to\infty} \frac{X_1(t)}{t}=\lambda_1(\mu_{23})<0 \right\}=1.
	$$
	\item[(5)]
		If $\lambda_1(\mu_{23})<0, \lambda_2(\mu_{13})<0$ and $\R^{13,\circ}$ and $\R^{23,\circ}_+$ are accessible from any $\bx\in\R^{3,\circ}_+$ then
	$$
	\PP_{\bx,i}\left\{\lim_{t\to\infty} \frac{X_1(t)}{t}=\lambda_1(\mu_{23})<0 \right\}+\PP_{\bx,i}\left\{\lim_{t\to\infty} \frac{X_2(t)}{t}=\lambda_2(\mu_{13})<0 \right\}=1.
	$$
\end{enumerate}

\end{thm}
\begin{proof}
	Let $\bar p_1,\bar p_2, \bar p_3>0$ such that
	\begin{equation}\label{barp}
		\bar p_1\lambda_1(\mu)-\bar p_2\lambda_2(\mu)+\bar p_3\lambda_3(\mu)>0,
	\text{for any }\mu\in\{\bdelta^*,\mu_1,\mu_{13}\}.
	\end{equation}
	
	Define $$\bar V(\bx)=\frac{x_2^{\bar p_2}}{x_1^{\bar p_1}x_3^{\bar p_2}}.$$
	As in the proof of Lemma \ref{lm3.1},  we can show that, for any $\bx\in\R^{13,+}$ and $\|\bx\|\leq M$, we have
	$$
	\frac1{\bar T}\int_0^{\bar T}\E_{\bx,j}\left(-\bar p_1 f_1(\BX(t),\xi(t)-\bar p_3 f_3(\BX(t),\xi(t)+\bar p_2 f_2(\BX(t),\xi(t)\right)dt<-\bar\rho<0
	$$
	for some $\bar\rho>0$, $\bar T>0$.
	Next, we can show as in Proposition \ref{prop2.1} that
	\begin{equation}\label{thm4.1-e2}
	\E_{\bx,j} \bar V(\BX(\bar T))\leq \bar\kappa \bar V( \bx), \text{ for all } \bx\in\R^{3,\circ}_+: x_2<\bar\delta, \|\bx\|\leq M
	\end{equation}
for some $\bar\kappa\in(0,1),\bar\delta>0$.
	Define $$U(\bx)=\bar V(\bx)\wedge \frac{\bar\delta^{\bar p_2}}{M^{\bar p_1+\bar p_2}}$$
	and
	$$\eta:=\inf\{k\geq0: \bar V(X(k\bar T))>\bar v_M \bar\theta^{k-1}\}.$$
	
	From \eqref{thm4.1-e2}, we have
\begin{equation}\label{thm4.1-e3}
	\PP_{\bx,j}\{\bar V(X(\bar T))>\varsigma\}\leq \frac{\E_{\bx,j} \bar V(\BX(\bar T))}{\varsigma}\leq
	\frac{\bar \kappa}{\varsigma} \bar V(\bx).
\end{equation}
	In particular, we have
	$$
	\PP_{\bx,j}\{\eta=1\}\leq
	\frac{\bar \kappa}{\bar v_M\bar\theta} \bar V(\bx).
	$$
	Similarly, using the Markov property of $(\BX(t),\xi(t))$ and \eqref{thm4.1-e3}, we have
	$$
	\begin{aligned}
	\PP_{\bx,j}\{\eta= 2\}=&\E_{\bx,j}\left[\1_{\{\eta>1\}}\PP_{\BX(\bar T), \xi(\bar T)}\{\eta=2\}\right]\\
	\leq&\E_{\bx,j}\left[\1_{\{\eta>1\}}
	\frac{\bar \kappa}{\bar v_M \bar\theta^2}\bar V(\BX(\bar T))\right]\\
	\leq & \frac{\bar\kappa^2}{\bar v_M\bar\theta^2} \bar V(\bx).
	\end{aligned}
	$$
	Continuing this way, we can show that
	$$\PP_{\bx,j}\{\eta<\infty\}=\sum_{k=1}^\infty \PP_{\bx,j}\{\eta=k\}\leq  \frac{\bar V(\bx) }{\bar v_M}\sum_{k=1}^\infty \frac{\bar\kappa^k}{\bar\theta^k}\leq \frac{\bar V(\bx) }{\bar v_M}\frac{\bar\theta}{\bar\theta-\bar\kappa}.$$
	
	This easily implies that if $\bar V(\bx)$ is sufficiently small then
	\begin{equation}\label{thm4.1-e4}
	\PP_{\bx,i}\left\{\lim_{k\to\infty}X_3(k\bar T)=0\right\}>1-\eps
	\end{equation}
	On the other hand, since $\BX(t)$ lives in a compact space, and the coefficients of \eqref{2prey_pdmp} are locally Lipschitz, there exists $\bar K>0$ such that
	$X_3(t)\leq \bar K X_3(k\bar T)$ for any $t\in (k\bar T,(k+1)\bar T)$.
	As a result,
		$$
	\PP_{\bx,i}\left\{\lim_{k\to\infty}X_3(k\bar T)=0\right\}>1-\eps.
	$$
	
	Finally, to obtain the exact convergence rate, we use the fact that any weak limit of the random occupation measure $\wdt\Pi_{t_k}:=\frac1t\int_0^t\1_{\{(\BX(s),\xi(s))\in\cdot\}}ds$ must be almost surely an invariant measure of $(\BX(t),\xi(t))$. If $X_3(t)$ converges to 0 then the weak limit must be an invariant measure on $\R_+^{13,\circ}\times \M$. Suppose with a positive probability, there exists a random sequence $\{t_k\}$ such that the limit of $\wtd \Pi_{t_k}$ is of the form $a_1\bdelta^*\times\pi+a_2\mu_1+a_3\mu_{13}$ with $a_1>0$ or $a_2>0$, then we show this leads to a contradiction as follows.
		We have from the weak convergence that
		$$
		\lim_{k\to\infty}\frac{\ln X_1(t_k)}{t_k}=\lim_{k\to\infty} \lambda_1(\wdt\Pi_{t_k})=\lambda_1(a_1\bdelta^*\times\pi+a_2\mu_1+a_3\mu_{13})=a_1\lambda_1(\bdelta^*\times\pi)
		$$
		because $\lambda_1(\mu_1)=0,\lambda_1(\mu_{13})=0$.
	Since $\lambda_1(\bdelta^*\times\pi)>0$ we must have $a_1=0$ otherwise
	$	\lim_{k\to\infty}\frac{\ln X_1(t_k)}{t_k}=\infty$, which contradicts the fact that the solution is bounded.
	
	Once we proved that $a_1=0$, we have
		$$
	\lim_{k\to\infty}\frac{\ln X_3(t_k)}{t_k}=\lim_{k\to\infty} \lambda_3(\wdt\Pi_{t_k})=\lambda_3(a_2\mu_1+a_3\mu_{13})=a_2\lambda_3(\mu_1)
	$$
	since $\lambda_3(\mu_{13})=0$.
	The fact that $X_3(t)$ is bounded implies that $a_2=0$ as well.
	
	As a result, we proved that the only weak limit of $\wtd\Pi_t$, if $X_3(t)$ converges to $0$, is $\mu_{13}$.
	Because of this uniqueness, we have
		$$
	\lim_{t\to\infty}\frac{\ln X_2(t)}{t}=\lim_{t\to\infty} \lambda_2(\wdt\Pi_{t})=\lambda_2(\mu_{13})
	$$
	for almost all trajectories satisfying $\lim_{t\to\infty}X_3(t)=0$.
	
	Combining this conclusion and \eqref{thm4.1-e4} completes our proof for part (1). Part (2) is similar.
	
	For parts (3), (4) and (5), we combine the result from part (1), the accessibility of the boundary and \cite[Lemma 3.1]{benaim2015qualitative} to obtain that
	$$
	\PP_{(\bx_0,i)}\left(\lim_{t\to\infty}\dist(\BX(t),\partial\R^{3}_+)=0\right)>0.
	$$
	This implies that there is no invariant measure on $\R^{3,\circ}_+\times\M$.
	As a result, any weak-limit of $\wdt\Pi_{t}(\cdot):=\frac1t\int_0^t \1_{\{(\BX(s),\xi(s))\in\cdot\}}ds$
	is an invariant measure on the boundary $\partial\R^{3,\circ}_+\times\M$.
	This can be used in conjunction with a standard contradiction argument \citep[Lemma 5.8]{HN16} to obtain the claims in parts (3), (4), and (5).
\end{proof}

\section{Examples}
In this section we showcase our theoretical results in two specific illuminating examples. For the deterministic system, without switching, corresponding to fixing $\xi(t)=j\in\M, t\geq 0$, if $b_1,b_2<1$ coexistence for the prey ecosystem $(X_1,X_2)$ is impossible in the absence of the predator. However, if $e_1(j)r>d$ and $e_2(j)r>d$ and $$\lambda_2(\delta_{13},j)=r-\frac{d}{e_1(j)}b_2-\left(r-\frac{d}{e_1(j)}\right)\frac{c_2(j)}{c_1(j)}>0,$$
$$\lambda_1(\delta_{23},j)=r-\frac{d}{e_2(j)}b_1-\left(r-\frac{d}{e_2(j)}\right)\frac{c_1(j)}{c_2(j)}>0,$$ where $(\delta_{13},j)$ is the point mass at the unique equilibrium of $(X_1, X_3)$ in environment $j$ on $R_+^{13,\circ}$, then the three-species ecosystem $(X_1, X_2 ,X_3)$ exhibits coexistence.

We will study how the random switching can change the longterm behavior of such ecosystems.

\vspace{1em}

\begin{exam}\label{ex5.1}

Consider the parameters

\[
\begin{cases}
r = 1,\tab d = 0.1\\
b_1 = 0.55,\tab b_2=  0.95\\
c_1(1) = 0.15,\tab c_1(2) = 0.4\\
c_2(1) = 0.178,\tab c_1(2) = 0.45\\
e_1(1) = 0.6,\tab e_1(2)  =0.85\\
e_2(1) = 0.45,\tab e_2(2) = 0.15.\\
\end{cases}
\]
Then
\[
\begin{cases}
\lambda_2(\delta_{13},1)\approx-0.0667, \\
\lambda_2(\delta_{13},2)\approx-0.05, \\
\lambda_2(\delta_{23},1) \approx0.185,\\
\lambda_2(\delta_{23},2) \approx 0.3111,\\
\end{cases}
\]
and
\[
\begin{cases}
r-\frac{d}{\bar{e_1}}b_2-(r-\frac{d}{\bar{e_1}})\frac{\bar{c_2}}{\bar{c_1}}\approx0.0022,\\
r-\frac{d}{\bar{e_2}}b_1-(r-\frac{d}{\bar{e_2}})\frac{\bar{c_1}}{\bar{c_2}}\approx0.1742,
\end{cases}
\]
where we set $\bar g=(g(1)+g(2))/2$ for $g=c_1,c_2,e_1,e_2$.
When the switching between the two environments is fast with equal rates $1\to2$ and $2\to1$, standard averaging arguments show that
$\lambda_2(\mu_{13})\approx r-\frac{d}{\bar{e_1}}b_2-(r-\frac{d}{\bar{e_1}})\frac{\bar{c_2}}{\bar{c_1}}$
and $\lambda_1(\mu_{23})\approx r-\frac{d}{\bar{e_2}}b_1-(r-\frac{d}{\bar{e_2}})\frac{\bar{c_1}}{\bar{c_2}}$.

As a result, the equilibrium point on the boundary $\R^{13,\circ}_+$ is asymptotically stable for both deterministic systems corresponding to state 1 and state 2. This shows that in the deterministic systems prey 2 goes extinct. However, with switching we have $\lambda_2(\mu_{13})>0$ and $\lambda_1(\mu_{23})>0$. By Theorem \ref{mthm-2} the three species coexist and converge to the unique invariant measure $\pi$ on $\R^{3,\circ}_+$ (see Figure \ref{fig:fig3}).
\end{exam}

\begin{figure}
\begin{subfigure}{0.48\textwidth}
    \includegraphics[width=0.95\textwidth]{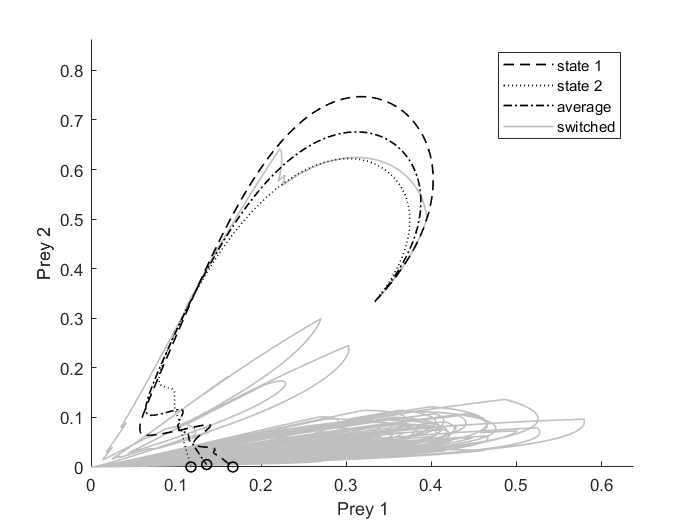}
    \label{fig:fig0}
\end{subfigure}\tab \begin{subfigure}{0.48\textwidth}
    \includegraphics[width=0.95\textwidth]{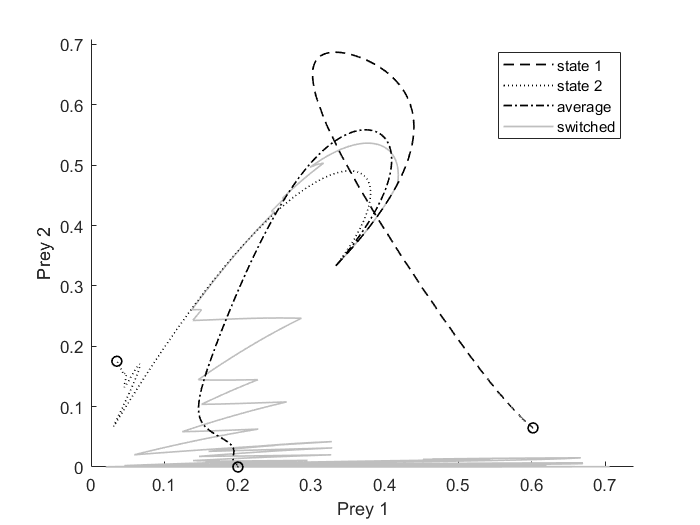}
   \label{fig:fig1}
\end{subfigure}
    \caption{Trajectories in prey 1 - prey 2 phase space. All simulations in a given panel have the same initial conditions. Small circles denote the fixed points for the various vector fields. Left panel: (Example \ref{ex5.1}) In each fixed environmental state prey 2 goes extinct. Switching makes all three species coexist. Right panel: (Example \ref{ex5.2}) In each fixed environmental state the three species coexist. Prey 2 goes extinct in the switched system.}
\label{fig:fig3} \end{figure}

\begin{figure}
\begin{subfigure}{0.45\textwidth}
    \includegraphics[width=0.95\textwidth]{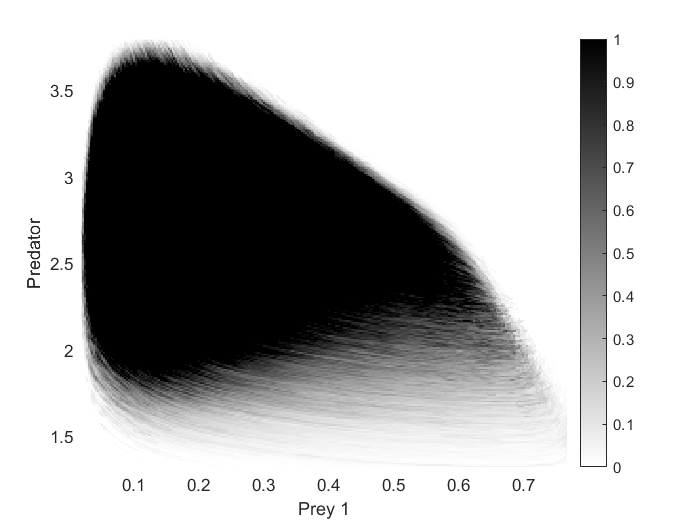}
    \caption{State 1}
    \label{fig:fig23}
\end{subfigure}\tab \begin{subfigure}{0.45\textwidth}
    \includegraphics[width=0.95\textwidth]{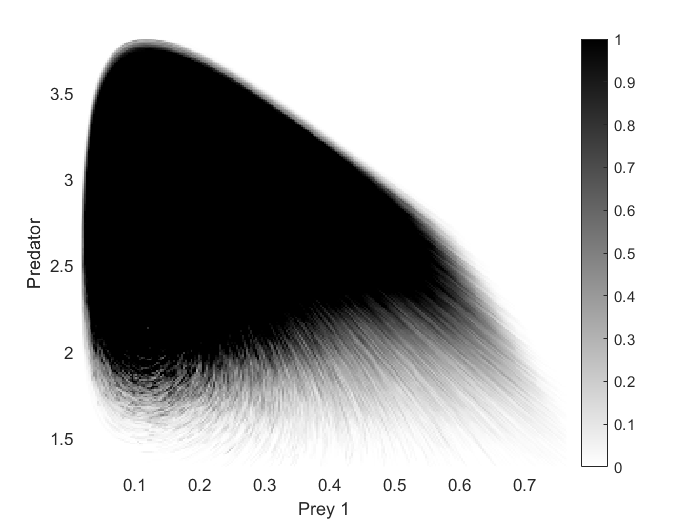}
   \caption{State 2}
   \label{fig:fig21}
\end{subfigure}
    \caption{(Example \ref{ex5.2}) The joint density of $X_1=\text{Prey 1}$ and $X_3=\text{Predator}$ in state 1 and state 2 was simulated 100 times on the time interval $[0,10000]$ for a solution $(X_1,X_2,X_3)$ initial values $(2/3,2/3,3/2).$ The occupation measure for the switched system converges exponentially fast to the absolutely continuous invariant measure on $\R^{13,\circ}_+$. }
\label{fig:fig2} \end{figure}

\begin{exam}\label{ex5.2}

Consider the parameters

\[
\begin{cases}
r = 1,\tab d = 0.1\\
b_1 = 0.9,\tab b_2=  0.5\\
c_1(1) = 0.15,\tab c_1(2) = 0.4\\
c_2(1) = 0.28,\tab c_1(2) = 0.4\\
e_1(1) = 0.15,\tab e_1(2)  =0.85\\
e_2(1) = 0.15,\tab e_2(2) = 0.4.\\
\end{cases}
\]

Then
\[
\begin{cases}
\lambda_1(\delta_{13},1)\approx0.1333, \\
\lambda_1(\delta_{13},2)\approx0.0667, \\
\lambda_1(\delta_{23},1) \approx0.6643,\\
\lambda_1(\delta_{23},2) \approx 0.0333,\\
\end{cases}
\]
and
\[
\begin{cases}
r-\frac{d}{\bar{e_1}}b_2-(r-\frac{d}{\bar{e_1}})\frac{\bar{c_2}}{\bar{c_1}}\approx-0.1114,\\
r-\frac{d}{\bar{e_2}}b_1-(r-\frac{d}{\bar{e_2}})\frac{\bar{c_1}}{\bar{c_2}}\approx 0.2483.
\end{cases}
\]

This shows that the equilibrium point in the interior $\R^{3,\circ}_+$ is asymptotically stable for both deterministic systems corresponding to state 1 and state 2. The three species coexist in both environments, if there is no randomness.
However, when the switching is fast, one has $\lambda_2(\mu_{13})\approx r-\frac{d}{\bar{e_1}}b_2-(r-\frac{d}{\bar{e_1}})\frac{\bar{c_2}}{\bar{c_1}}<0$ and $\lambda_1(\mu_{23})\approx r-\frac{d}{\bar{e_2}}b_1-(r-\frac{d}{\bar{e_2}})\frac{\bar{c_1}}{\bar{c_2}}>0$. Using Theorem \ref{t:extinction} we see that in the random system, prey 1 and the predator persist, while prey 2 can go extinct with a large probability when it starts at a small initial density (see Figure \ref{fig:fig3} and Figure \ref{fig:fig2}).
\end{exam}

%Constructing the Lyapunov function
%$$
%V(\bx)=\dfrac{1+\bc^\top\bx}{\left((x_1+x_2)^{p_0}x_1^{p_1}x_2^{p_2}x_3^{p_3}\right)^\gamma}
%$$
%where $\gamma$ is sufficiently small.
%
%
%
%
%As a result, there exists an $\hat\eps>0$, $\hat\gamma>0$ such that
%\begin{equation}\label{lowerb}
%\begin{aligned}
%\dfrac{\sum c_ix_if_i(\bx)}{1+\bc^\top\bx}&+p_0\left(\dfrac{x_1f_1(\bx)+x_2f_2(\bx)}{x_1+x_2}-\dfrac{\sum_{i,j=1}^2\sigma_{ij}x_ix_jg_i(\bx)g_j(\bx)}{2(x_1+x_2)^2}\right)\\-&\sum\left(p_if_i(\bx)-\dfrac{p_ig_i^2(\bx)\sigma_{ii}}2\right)<-2\hat\gamma \text{ for } \|\bx\|<\hat\eps.
%\end{aligned}
%\end{equation}
%\begin{equation}\label{e2.1}
%\begin{aligned}
%&\dfrac{\delta_0^2}2p_0^2\dfrac{\sum_{i,j=1}^2\sigma_{ij}x_ix_jg_i(\bx)g_j(\bx)}{(x_1+x_2)^2}+\delta_0p_0\sum_{i=1}^2\sum_{j=1}^3\sigma_{ij}\dfrac{x_ig_i(\bx)}{x_1+x_2}g_j(\bx)\\
%&+\delta_0\sum\dfrac{c_ip_ix_i\sigma_{ij}g_i(\bx)g_j(\bx)}{(1+\sum c_ix_i)}+\dfrac{\delta_0}2\sum p_ip_j\sigma_{ij}g_i(\bx)g_j(\bx)\\
%&-\delta_0p_0\sum_{i=1}^2\sum_{j=1}^3\sigma_{ij}\dfrac{x_ic_jx_jg_i(\bx)g_j(\bx)}{(x_1+x_2)(1+c^\top x)}\leq \hat\gamma \text{ for } \|\bx\|<\hat\eps.
%\end{aligned}
%\end{equation}
%
%
%Applying the arguments in our AAP paper, We can show that
%$$
%\E_\bx V(\BX(T))^\theta\leq r V(\bx), r\in (0,1)
%$$
\textbf{Acknowledgments:} The authors acknowledge
support from the NSF through the grants DMS-1853463 for
Alexandru Hening and DMS-1853467 for Dang Nguyen.
\bibliographystyle{agsm}
\bibliography{LV}
\end{document}